\documentclass{article}
\usepackage{graphicx} 
\usepackage{amsmath}
\usepackage{amsthm}
\usepackage{amssymb}
\usepackage[dvipdfm, bookmarks, colorlinks, breaklinks, pdftitle={Gonzalo Fiz Pontiveros - AFR Proposal},pdfauthor={Gonzalo Fiz Pontiveros}]{hyperref}  
\hypersetup{linkcolor=blue,citecolor=blue,filecolor=black,urlcolor=blue} 

\numberwithin{equation}{section} 

\theoremstyle{plain}
\newtheorem{corollary}{Corollary}[section]
\newtheorem{proposition}{Proposition}[section]
\newtheorem{theorem}{Theorem}[section]
\newtheorem{lemma}{Lemma}[section]

\theoremstyle{definition}
\newtheorem{definition}{Definition}[section]

\theoremstyle{remark}
\newtheorem{obs}{Observation}[section]
\newtheorem*{claim}{Claim}
\newtheorem{remark}{Remark}

\newcommand{\Z}[1]{\mathbb{Z}_{#1}}
\newcommand{\prob}[1]{\mathbb{P}(#1)}
\newcommand{\exn}[2]{\mathbb{E}_{#1} \left ( #2 \right )} 
\newcommand{\case}[1]{\emph{Case #1},}

\begin{document}

\title{Freiman rank of random subsets of $\Z{N}$}
\author{Gonzalo Fiz Pontiveros}
\date{\today}

\maketitle
\begin{abstract}
Let $A$ be a random subset of $\mathbb{Z}_{N}$ obtained by including each element of $\mathbb{Z}_{N}$ in $A$ independently with probability $p$. We say that $A$ is \emph{linear} if the only Freiman homomorphisms are given by the restrictions of    functions of the form $f(x)= ax+b$. For which values of $p$ do we have that  $A$ is linear with high probability as $N\to\infty$ ? 

First, we establish a geometric characterisation of linear subsets. Second, we show that if $p=o(N^{-2/3})$ then $A$ is not linear with high probability whereas if $p=N^{-1/2+\epsilon}$ for any $\epsilon>0$ then $A$ is linear with high probability.  \end{abstract}

\section{Introduction}

Freiman's structure theory of set addition constitutes 
now one of the most general and powerful tools in additive combinatorial number theory. The essential concept of this theory is what is now known as  \emph{Freiman homomorphism} in the literature. \\

\noindent Let $A\subseteq \Z{N}$ and let $\phi:A \to \Z{N}$ be some function.\\

\begin{definition}
 \emph{We say that $f$ is a Freiman homomorphism if whenever a quadruple $a,b,c,d \in A$ satisfies $a-b=c-d$ then $\phi(a)-\phi(b)=\phi(c)-\phi(d)$.}	\\
 \end{definition}

 Clearly if $f:\Z{N}\to \Z{N}$ is of the form $f(x)= ax+b$, that is the translate of a group homomorphism, then the restriction $f|_{A}$ is a Freiman homomorphism. We will refer to functions of the above form as \emph{linear}.\\
 
 For the sake of simplicity assume now that $N$ is a  prime. Then it is easy to see that the space of Freiman homomorphisms from $A$ to $\Z{N}$, denoted $\text{Hom}_{F}(A,\Z{N})$, is a vector space over the field $\mathbb{F}\cong \Z{N}$.We consider the notion of \emph{Freiman rank or Freiman dimension} of $A$: 

 $$rank(A)=\dim{\text{Hom}_{F}(A,\Z{N})}-1$$

Observe that $1\leq rank(A)\leq |A|-1$. The intuition here is that the size of the rank should give some form of measure of the additive structure of $A$. For example if $A$ is an arithmetic progression i.e $A=\{a_{0}+ i\cdot d\; :\; 0\leq i\leq l-1\}$ then the only Freiman homomorphisms are given by the restrictions of linear functions to $A$ and hence $rank(A)=1$. If $A\subseteq \Z{N}$ has Freiman rank $1$ we say that the set $A$  is \emph{linear}.\\

 On the other side of the additive spectrum, we could pick $A$ to be a \emph{Sidon set}; that is a set where the only quadruples $(a,b,c,d) \in A^{4}$ such that $a+b=c+d$ are the trivial ones. The classical example of a Sidon set is the set $\{1,2,4,\ldots, 2^{k}: k< \log_{2}{(N/2)}\}$. In this case the restrictions for a function to be a Freiman homomorphism are essentially empty so in fact any function $\phi: A\to \Z{N}$ is a Freiman homomorphism and therefore the Freiman rank is as large as it can be, namely $rank(A)=|A|-1$.\\

It is possible to extend the definition of Freiman rank for $N$ not prime or indeed any abelian group. However we will be only considering sets of Freiman rank 1 for which we can give a simple independent definition.\\

\begin{definition}\emph{ We say that a set $A\subseteq \Z{N}$ is linear if and only if the only Freiman homomorphism are given by the restrictions of linear functions}.\\
\end{definition}
 
\noindent We will now state the main result of this paper:
\begin{theorem}
\label{Main}
For any $\epsilon >0$, let $A$ be a random subset of $\Z{N}$ where each $x \in \Z{N}$ is chosen independently with probability $p=N^{-\frac{1}{2} + \epsilon}$. Then with high probability $A$ is a linear set. Furthermore if $p=o(N^{-2/3})$ then with high probability we may find non trivial Freiman homomorphisms on $A$.
\end{theorem}

\noindent We take the opportunity to give a quick proof of  the lower bound in Theorem $\ref{Main}$  We claim the following holds:

\begin{claim} 
Let $A\subset \Z{N}$  with $|A|\geq 3$ and suppose there exists some $x_{0} \in A$ such that no $(x,y,z)$ in $A^{3}$ satisfies $x+y=z+x_{0}$. Then we may construct a non-trivial Freiman homomorphism $f: A \to \Z{N}$.
\end{claim}
\begin{proof}
Set $B=A-x_{0}$. Since $B$ is simply a translate $A$ we will be done if we can construct a non-trivial Freiman homomorphism on $B$.  Note that $0 \in B$ and there are no triples $(x,y,z)$ in $B$ such that $(x+x_{0})+(y+x_{0})=(z+x_{0})+x_{0}$, which is to say $x+y=z$. In particular the only additive quadruples involving the element $0\in B$ are the trivial ones, that is of the form $0+x=x+0$. \\

\noindent Now define $f : B\to\Z{N}$ as
\begin{equation*}
f(x)=
\begin{cases}
1 \quad\text{if}\; x=0\\
0 \quad\text{otherwise}
\end{cases}
 \end{equation*}
Provided $|B|\geq 3$ the function $f$ is not the restriction of a linear function and furthermore it is easy to see that it indeed defines a Freiman homomorphism on $B$.
\end{proof}

\noindent Thus it is sufficient to show that for $p=o(N^{-2/3})$ we may find such $x_{0} \in A$ with high probability. Let $X$ be the random variable given by the number of additive quadruples in $A^{4}$. Clearly the number of $x \in A$ that are not possible candidates for the above $x_{0}$ is at most  $X$. 
By Markov's inequality we have that:
\begin{equation}
\prob{X\geq \frac{1}{2}Np}\leq 2\frac{1}{Np}\exn{}{X}\leq N^{2}p^{3}=o(1)
\end{equation}
since $\exn{}{X}\leq N^{3}p^{4}$. On the other hand the number of elements in $A$ is given by a $\text{Bin}(N,p)$ binomial random variable which we know to be strongly concentrated around the it's mean provided that $p=\omega(\frac{1}{N})$, in particular
\begin{equation*}
\prob{|A|\leq \frac{1}{2}Np}= o(1)\;,
\end{equation*}
therefore with high probability $|A|>X$ and it follows that there must exist $x_{0} \in A$ with no triple $(x,y,z)$ in $A^{3}$ satisfying $x+y=z+x_{0}$.

\section{Setting and basic observations}

Let $A \subset \Z{n}$ and $f: A \rightarrow \Z{N}$ be a Freiman homomorphism. For the sake of simplicity, we will assume throughout the following sections that $A-A =\Z{N}$ and that $0 \in A$. These assumptions become immaterial when we return to the probabilistic setting since, provided $p=\omega(N^{-1/2})$, then $A-A=\Z{N}$ with high probability and $A$ is linear if and only any translate of $A$ is also linear.

\begin{definition}
The \emph{induced function of f}, $\phi_{f}: \Z{N} \to \Z{N}$ is given by: 
\begin{equation*}
\phi_{f}(d)= f(x+d)-f(x)\; \text{where}\; x, x+d \in A 
\end{equation*}
Note that $\phi_{f}$ is well defined since $f(x+d)-f(x)=f(y+d)-f(y)$ for any $x,y \in A$ since $f$ is a Freiman homomorphism. We will  refer to the induced function simply as $\phi$ unless further clarification is required.
\end{definition}
Here is the key property of the induced function:

\begin{proposition}
A Freiman homomorphism $f$ on $A$ is linear if and only  if the induced function $\phi$
is a group homomorphism.
\end{proposition}

\begin{proof}
Since $\phi$ is a group homomorphism it satisfies that $\phi(d+d')=\phi(d)+\phi(d')$ for all $d,d' \in \Z{N}$ . Hence $f$ is linear as for any $x \in A$ $$f(x)-f(0)=\phi(x)=\phi(x-1)+\phi(1)=\ldots=x\phi(1)$$ The converse is clearly true.
\end{proof}
We are interested in understanding which structural properties of $A$ would guarantee that $\phi$ is a linear function. A first simple observation is the following:

\begin{obs}
\label{triples}
Whenever $A$ contains a triple of the form $x,x+d,x+d+d'$ it follows that  $\phi(d+d')=\phi(d)+\phi(d')$  since 
\begin{align*}
f(x+d+d')-f(x)& =&  \big (f(x+d+d')-f(x+d)\big)&+\big(f(x+d)-f(x)\big)\\
& =&  \phi(d') &+\phi(d)
\end{align*}
\noindent We will say, for convenience,  that such a pair $(d,d')$ is \emph{additive}.
\end{obs}

\noindent We also have a sort of converse  for this observation:
\begin{proposition}
\label{incl}
Suppose that $\psi : \Z{N} \to \Z{N}$ is a function such that
\begin{eqnarray*}
\psi(d_{1}+d_{2})&=&\psi(d_{1})+\phi(d_{2})
\end{eqnarray*}
whenever the pair $(d_{1},d_{2})$ is additive. Then there exists a Freiman homomorphism $f: A \to \Z{N}$ such that $\phi_{f}=\psi$
\end{proposition}

\begin{proof}
The set of additive pairs is invariant under translations and so are Freiman homomorphisms so we may assume  without loss of generality that $0 \in A$.
 
Set $f:A\to\Z{N}$ to be $\psi |_{A}$. Firstly we need to check that that $\psi$ preserves additive quadruples and hence is a  Freiman homomorphism: note that each quadruple may be expressed as $(x,x+d_{1},x+d_{2},x+d_{1}+d_{2}) \in A^{4}$. Now we exploit the fact that $\psi$ satisfies all additive pairs:
\begin{eqnarray*}
 \psi(x+d_{i})&=&\psi(x)+\psi(d_{i}) \mbox{ as } (0,x,x+d_{i}) \in A^{3}\\  
 \psi(x+d_{1}+d_{2})&=& \psi(x) + \psi(d_{1}+d_{2}) \mbox{ as } (0,x,x+d_{1}+d_{2}) \in A^{3}\\
 \psi(d+d')&=&\psi(d) + \psi(d') \mbox{ as } (x,x+d_{1},x+d_{1}+d_{2}) \in A^{3}. 
 \end{eqnarray*}
Therefore $\psi(x+d_{1}+d_{2})+\psi(x)= \psi(x)+\psi(d_{1})+\psi(d_{2})+\psi(x)=\psi(x+d_{1})+\psi(x+d_{2})$ as required. Now recall that $\phi_{f}(d)= f(x+d)-f(x)$ where $x,x+d \in A$ and so $\phi_{f}(d)=\psi(x+d)-\psi(x)=\psi(d)$ since the triple $(0,x,x+d)\in A^{3}$ 
\end{proof}

Hence the simplest condition one could ask to force $\phi$ to be a homomorphism is that for each pair of differences $(d,d') \in \Z{N}^{2}$ we can find a triple of the form $(x,x+d,x+d+d') \in A^{3}$ or in other words  that every pair $(d,d')$ is additive. However, if we have the probabilistic setting in mind, the probability of finding such a triple for $d,d'\neq 0$ is at most $Np^{3}$ by the trivial union bound and thus for $p \ll N^{-1/3}$ this event will occur with small probability. If we are to have any hope of showing Theorem $\ref{Main}$ we ought to look at things in greater detail. 

The second observation is that in fact we need not ask for all pairs $(d,d')$ to be additive in order to conclude that $\phi$ is a homomorphism. For instance, if we are already given that the pairs $(3,1)$, $(2,1)$ and $(1,1)$ are additive we could deduce that $\phi(2+2)=\phi(3+1)=\phi(2)+\phi(1)+\phi(1)= \phi(2)+\phi(2)$ which is to say the pair $(2,2)$ is additive. 

This suggests that we look at how much information we can extract from additive pairs. More explicitly, we would like to answer the following question:  given a fixed set of additive pairs, which equalities other than the trivial ones may we deduce from it? It turns out that this problem can be interpreted as determining whether a given word is the trivial element in a certain group presentation; it is not surprising that in order to tackle it, it will be helpful to introduce a notion very much analogous to that of a \emph{Cayley complex}.

\section{A topological space}
\begin{definition}
Let $\mathcal{C}_{A}$ be the $2$-dimensional cell-complex constructed from $A$ as follows: \\

\noindent$\bullet$ The vertices of $\mathcal{C}_{A}$ are simply the elements of $\Z{N}$ \\

\noindent$\bullet$ The $1$-cells are given by all directed edges $x\to y$ for $x\neq y$\\
\phantom y and are labeled by  $\overline{d}$ where $d=y-x$\\

\noindent$\bullet$ Whenever three edges with labels $\overline{d_{1}}, \overline{d_{2}}, \overline{d_{3}}$ form an oriented triangle $T$ in $C_{A}$\\
\phantom y(which implies that $d_{1}+d_{2}+d_{3}=0$) and there exists a triple of the form\\ \phantom y$(x,x+d_{1},x+d_{1}+d_{2}) \in A^{3}$ we add a $2$-cell with $T$ as its boundary and we\\
 \phantom y label it $[\overline{d_{1}},\overline{d_{2}},\overline{d_{3}}]$ \\

\end{definition}

\begin{remark}
The $1$-skeleton of $\mathcal{C}_{A}$ is the complete Cayley graph on $\Z{N}$. The use of $\overline{d}$ to label the edges may seem strange to the reader; the reason behind it is that we will be considering the chain group given by the edge labels and we wish to make clear the distinction between $d$, an element of $\Z{N}$, and $\overline{d}$, an element of the chain group.
\end{remark}
\begin{remark}
Usually when dealing with simplicial complexes one denotes a specific simplex by its vertex set but  since our definition of a $2$-cell is invariant under any translate it is more appropriate to use the labels of edges bounding it. Also we will make no distinction between a specific $2$-cell and its corresponding label since all the properties we are interested in are translation invariant and two distinct $2$-cells carry the same label if and only if there is a translation mapping the vertices of one to the other.\\
\end{remark}

\noindent The crucial aspect of the construction is the following:
\begin{proposition}
Whenever a $2$-cell $[\overline{d_{1}},\overline{d_{2}},\overline{d_{3}}]$ is in $\mathcal{C}_{A}$ then any induced function $\phi$ must satisfy.
\begin{equation}
\phi(d_{1})+\phi(d_{2})+\phi(d_{3})=0 \tag{$\star$}
\end{equation}
\end{proposition}
\begin{proof}
Since $[\overline{d_{1}},\overline{d_{2}},\overline{d_{3}}]$ is a $2$-cell in $C_{A}$ we know that $d_{1}+d_{2}+d_{3}=0$, which is to say $-d_{3}=d_{1}+d_{2}$ and that there exists a triple $(x,x+d_{1},x+d_{1}+d_{2})$ in $A^{3}$. 

By observation $\ref{triples}$ we can conclude that any induced function must satisfy $\phi(d_{1}+d_{2})=\phi(d_{1})+\phi(d_{2})$, furthermore, any induced function must also verify that $\phi(-d)=-\phi(d)$ and hence
\begin{eqnarray*}
\phi(d_{3})+\phi(d_{2})+\phi(d_{1})&=&\phi(-(d_{1}+d_{2}))+\phi(d_{2})+\phi(d_{1})\\
&=&-(\phi(d_{1})+\phi(d_{2}))+\phi(d_{2})+\phi(d_{1})\\
&=&0
\end{eqnarray*}

\end{proof}
\noindent We will say that $\alpha=(\overline{d_{1}},\overline{d_{2}},\dotso,\overline{d_{l}})$ is a cycle in $\mathcal{C}_{A}$ whenever $\partial(\alpha)=0$, where $\partial$ denotes the boundary function. Note that this coincides with the graph-theoretic notion of cycle.
\begin{proposition}
\label{hogy}
Let $\alpha=(\overline{d_{1}},\overline{d_{2}},\dotso,\overline{d_{l}})$ be a cycle in $\mathcal{C}_{A}$. If    $\alpha$ belongs to the trivial homology class then any induced function $\phi$ must satisfy
\begin{equation}
\phi(d_{1})+\phi(d_{2})+\dotso+\phi(d_{l}) = 0 
\end{equation}
\end{proposition}
\begin{proof}
A cycle is in the trivial homology class if and only if we can express it as the boundary of a collection of $2$-cells. Hence
\begin{eqnarray}
\label{cycles}
\begin{split}
\alpha &= \overline{d_{1}}+\ldots +\overline{d_{l}} = \partial \Big [\sum_{j}\sigma_{j}\Big] = \sum_{j}\partial \sigma_{j}\\
&=\sum_{j}\partial [\overline{d_{j1}},\overline{d_{j2}},\overline{d_{j3}}]=\sum_{j} \big ( \overline{d_{j1}}+\overline{d_{j2}}+\overline{d_{j3}} \big ) 
\end{split}
\end{eqnarray}
And so 
\begin{eqnarray*}
\phi(d_{1})+\phi(d_{2})+\dotso+\phi(d_{l})= \sum_{j}\Big (\phi(d_{j1})+\phi(d_{j2})+\phi(d_{j3})\Big)=0
\end{eqnarray*}
as each term in the brackets must add up to $0$ by $(\star)$.
\end{proof}
\begin{remark}
The additions in \eqref{cycles} take place in the chain group of paths, that is we only allow cancellations of the form $\overline{d}+(-\overline{d})=0$. Again since any induced function satisfies $\phi(-d)=-\phi(d)$ for any $d$ we are safe.\\
\end{remark}

\begin{corollary}
\label{hogy2}
Let $A$ be a subset of $\Z{N}$ such that $A-A=\Z{N}$, and let $\mathcal{C}_{A}$ be the cell-complex defined above. Suppose that $\mathcal{C}_{A}$ has a trivial first homology group. Then every Freiman homomorphism $f:A\to \Z{N}$ is the restriction to $A$ of a linear function. 
\end{corollary}
\begin{proof}
Any cycle in $\mathcal{C}_{A}$ has trivial homology class. In particular, for any choice of $d,d' \in \Z{N}$ the cycle $(d,d',-d-d')$ has trivial homology class and so by Proposition $\ref{hogy}$ we have that the induced function satisfies $\phi(d+d')=\phi(d)+\phi(d')$. Hence $\phi$ is a homomorphism and the result follows.
\end{proof}
What is the principle behind this construction? Recall that our ultimate aim is to understand what the space of possible induced functions $\phi$ looks like. We know that, whenever a pair $(d_{1},d_{2})$ is additive, any such a function must satisfy $\phi(d_{1}+d_{2})=\phi(d_{1})+\phi(d_{2})$ so it is natural to turn our attention to this, a priori, larger space of functions:  \\

\noindent Let $\mathcal{F}$ be the space of all functions $\phi :\Z{N}\to \Z{N}$ such that $\phi(0)=0$ and
\begin{eqnarray*}
\phi(d_{1}+d_{2})&=&\phi(d_{1})+\phi(d_{2})
\end{eqnarray*}
whenever the pair $(d_{1},d_{2})$ is additive.\\

\noindent This is a submodule (over $\Z{N}$) of the free module  $\mathcal{M} \cong \Z{N}^{N-1}$ of all functions $f: \Z{N}\rightarrow \Z{N}$ such that $f(0)=0$. We may take as a basis the elements $e_{1},\ldots,e_{N-1}$ where the $e_{j}:= \mathbb{I}_{\{j\}}$ are the indicator functions taking the value $1$ at $j$ and $0$ otherwise. \\

\noindent Consider the subgroup
 
\begin{eqnarray*}
\mathcal{B}:= \Big \langle e_{d_{1}}+e_{d_{2}}+e_{d_{3}} \; : \;[\overline{d_{1}},\overline{d_{2}},\overline{d_{3}}]\in  \mathcal{C}_{A} \Big \rangle
\end{eqnarray*}

For $\phi, \psi \in \mathcal{M}$ we may define a non-degenerate symmetric bilinear form, analogous to an inner product, by 
$$\langle \phi, \psi \rangle = \sum_{x \in \Z{N}}\phi(x)\psi(x)$$
 It is easy to see that $\phi \in \mathcal{F}$ if and only if $\langle\phi, f\rangle = 0$ for all $f \in \mathcal{B}$ or, using the vector space notation, $\mathcal{F}=\mathcal{B}^{\perp}$. We will make use of the fact that 
 \begin{equation}
 \label{perp}
 |\mathcal{M}|=|\mathcal{B}^{\perp}||\mathcal{B}| 
 \end{equation}
 
  On the other hand we have that $\mathcal{B}$ is, by construction, isomorphic to the group generated by the boundaries of the the $2$-cells of $\mathcal{C}_{A}$. What is the group generated by the cycles? 
  
  Let $G\cong \Z{N}^{N-1}$ be the free module generated by the edge labels and consider the homomorphism $\psi :G \to \Z{N}$ such that $\psi(\overline{d})=d$ for all $\overline{d}\in \mathcal{C}_{A}$. The map is clearly surjective and an element $x \in G$ is a cycle in $\mathcal{C}_{A}$ if and only if $x \in \text{ker}\; \psi$. Hence, using the classification theorem of abelian groups, it follows that  $\text{ker}\;\psi\cong \Z{N}^{N-2}$. 

Therefore the first homology group of the cell-complex $\mathcal{C}_{A}$ is isomorphic to
$\Z{N}^{N-2}/\mathcal{B}$. \\

 This yields a more abstract proof of Proposition $\ref{hogy}$ : If $\mathcal{C}_{A}$ has trivial first homology group then $\mathcal{B}\cong \Z{N}^{N-2}$. Now, since $|\mathcal{M}|=|\mathcal{B}^{\perp}||\mathcal{B}|$ it follows that $|\mathcal{B}^{\perp}|=N$ and so $\mathcal{F}= \mathcal{B}^{\perp}\cong \Z{N}$ as we already now that $\mathcal{F}$ certainly contains the subgroup of linear functions. Hence $\mathcal{F}$ is  precisely the subgroup of linear functions.\\
 
The advantage of this formulation is that it is now simple to show that the converse of Proposition $\ref{hogy}$ also holds:
\begin{proposition}
\label{conv}
Let $A\subset \Z{N}$ and $\mathcal{C}_{A}$ be as above. Suppose that $\mathcal{C}_{A}$ has a non trivial first homology class. Then there exist a non linear Freiman homomorphism $f:A\to \Z{N}$.
\end{proposition}
\begin{proof}
If $\mathcal{C}_{A}$  has non trivial first homology group then $\mathcal{B}$ is a proper subgroup of $\Z{N}^{N-2}$. In particular $|\mathcal{B}|<N^{N-2}$ and from \eqref{perp} it follows that $|\mathcal{F}|=|\mathcal{B}^{\perp}|>N$. Hence, using Proposition $\ref{incl}$ we can conclude that the set of Freiman homomorphisms $f:A\to \Z{N}$ with $f(0)=0$ is strictly greater than $N$ and thus there must exist one that is not linear. 

\end{proof}
\begin{remark}
If we consider homomorphisms  $f:\Z{N}\to \Z{}$ instead the above arguments carry through and give a stronger result: if the cell-complex $\mathcal{C}_{A}$ has trivial first homology group then the only Freiman homomorphisms are the constant functions. 	
\end{remark}

\section{A family of  surfaces}

The previous section has provided us with an accurate geometric description of the structural properties that are necessary and sufficient for a set to have only linear Freiman homomorphisms. 

\vspace{3 pt}
\noindent This suggests the following strategy: 
\begin{enumerate}
\item Pick an orientable surface $\mathcal{H}$ together with a triangulation $\Delta(\mathcal{H})$ that has a triangle $T$  as boundary. 
\item Fix an oriented triangle in $[\overline{a},\overline{b},\overline{c}] \in \mathcal{C}_{A}$
\item Attempt to embed  $\Delta(\mathcal{H})$ in $\mathcal{C}_{A}$ in such way that $T$ gets mapped to $[\overline{a},\overline{b},\overline{c}]$  and any triangular face of $\Delta(\mathcal{H})$ gets mapped into some $2$-cell of $C_{A}$. 
\end{enumerate}

If we can do so then we have explicitly shown that the homology class of the oriented triangle $[\overline{a},\overline{b},\overline{c}]$ is the trivial one and hence, by Corollary $\ref{hogy}$, the pair $(b-a,c-b)$ is additive. The aim is to estimate carefully what is the probability that this process succeeds for all choices of $[\overline{a},\overline{b},\overline{c}]$. 
\vspace{4 pt}
\begin{remark}
We may assume henceforth that $a,b,c$ are distinct as it follows immediately from the definition and the fact that $A-A=\Z{N}$ that pairs of the form $(0,d)$ and $(-d,d)$  must be additive.
\end{remark}

For instance if we set $\mathcal{H}$ to be a single triangular face then embedding $\mathcal{H}$ is precisely demanding that the pair $(b-a,c-a)$ is additive. In other words, with this particular choice of $\mathcal{H}$, the strategy above is the same as trying to show that every pair is additive, which, as we have seen previously, is a sufficient 
condition to guarantee that all induced functions $\phi$ are linear.\\

The hope is that, by choosing more complex $\mathcal{H}$, we will be able to improve the range of values of $p$ for which this embedding strategy is successful with high probability.\\

\noindent We will now consider in detail a particular sequence of simplicial complexes:
\begin{align*}
\mathcal{H}_{0} & = [a,b,c]\\
\mathcal{H}_{1} &= [a,b,z][a,z,c][z,b,c]
 \end{align*}
and in general $\mathcal{H}_{i+1}$ is obtained from $\mathcal{H}_{i}$ by taking each $2$-simplex $[x_{1},x_{2},x_{3}] \in \mathcal{H}_{i}$ and subdividing into three new simplexes $[x_{1},x_{2},x][x_{1},x,x_{3}][x,x_{2},x_{3}]$.  

\begin{figure}[h]
\includegraphics[scale=0.7]{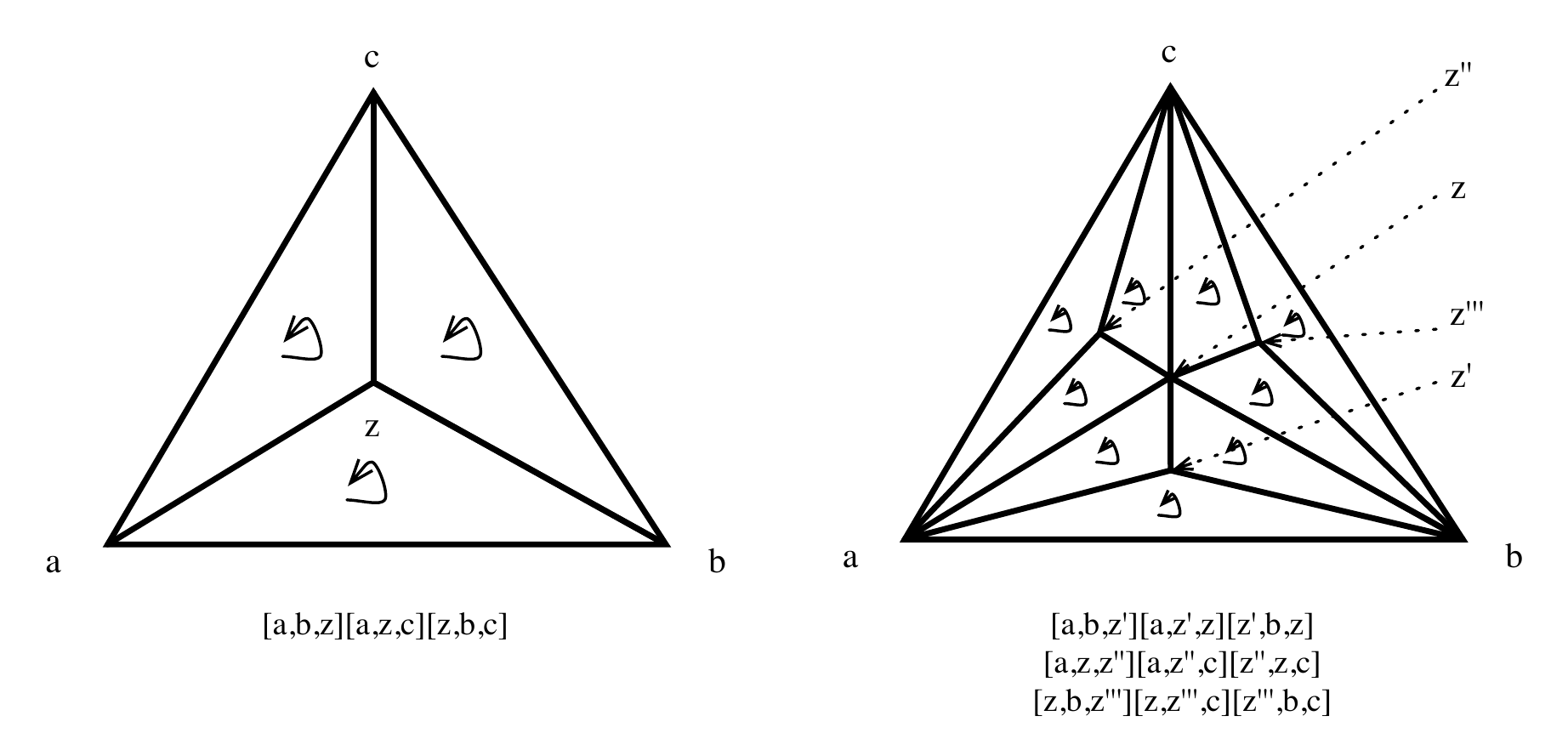}
\caption{Sketch of $\mathcal{H}_{1}$ and $\mathcal{H}_{2}$}
\label{fig1}
\end{figure}
\bigskip
\bigskip
We will show that the family $(\mathcal{H}_{i})_{i\geq 1}$ is a suitable family of simplicial complexes. That is to say that, for any $p=N^{-1/2+\epsilon}$ with $\epsilon>0$, we can find some $i$ (depending only on $\epsilon$) such that with high probability, for all oriented triangles $[\overline{a},\overline{b},\overline{c}]$, we can embed $\mathcal{H}_{i}$ in $\mathcal{C}_{A}$ with boundary $[\overline{a},\overline{b},\overline{c}]$. \\

In order to do this it will be helpful to express the above statement in analytic terms; it will allow us to bring in powerful tools of probability theory concerning the concentration of random variables about their mean.\\

Let us consider the case when $i=0$. The only way we can embed $\mathcal{H}_{0}$ in $\mathcal{C}_{A}$ with boundary $[\overline{a},\overline{b},\overline{c}]$ is if the oriented triangles is the boundary of a $2$-cell in $\mathcal{C}_{A}$ i.e. if and only if there exists a triple of the from $(x,x+a,x+a+b)\in A^{3}$. In terms of the characteristic function $1_{A}$ this statement is equivalent to 
$$\sum_{x\in \Z{N}}t_{x}t_{x+a}t_{x+a+b}>0$$
where $t_1 \ldots t_N$ are Boolean variables with
\begin{equation*}
t_{x}=1_{A}(x)=\begin{cases} 
1& \text{if $x\in A$},\\ 
0& \text{if otherwise} 
\end{cases} 
\end{equation*}

\noindent Once we have the first instance, it is simple to give a recursive expression for the subdivisions. 

\begin{lemma}
Define the family of polynomials

\begin{align*} 
\Lambda_{a,b,c}^{0}&= \sum_{ x \in \Z{n}} t_{x+a}t_{x+b}t_{x+c} \\
\Lambda_{a,b,c}^{i+1}&= \sum_{z \in \Z{n}}\Lambda_{a,b,z}^{i}\Lambda_{a,z,c}^{i}\Lambda_{z,b,c}^{i}
\end{align*} 

Then the  complex $\mathcal{H}_{i}$ can be embedded in $\mathcal{C}_{A}$  with boundary $[\overline{b-a},\overline{c-b},\overline{c-a}]$ if and only if  $$\Lambda_{a,b,c}^{i}[t_{1},\dotso,t_{N}] > 0$$
\end{lemma}
\begin{proof}
The lemma follows from a simple induction argument. Clearly $\Lambda_{a,b,c}^{0}>0$ if and only if there exists a triple $x+a,x+b,x+c \in A$ which is to say $[\overline{b-a},\overline{c-b},\overline{c-a}] \in \mathcal{C}_{A}$. Assume the result holds for $i$ and notice that $\mathcal{H}_{i+1}$  can be embedded in ${\mathcal{C}_{A}}$ with boundary $[\overline{b-a},\overline{c-b},\overline{c-a}]$ if and only if there exists some $z \in \Z{N}$ such that $\mathcal{H}_{i}$ can be embedded with boundary $[\overline{b-a},\overline{z-b},\overline{z-a}]$, $[\overline{z-a},\overline{c-z},\overline{c-a}]$ and $[\overline{b-z},\overline{c-b},\overline{c-z}]$. By the induction hypothesis this occurs if and only if $\Lambda_{a,b,z}^{i}\Lambda_{a,z,c}^{i}\Lambda_{z,b,c}^{i}>0$ for some $z \in \Z{N}$ i.e if and only if 

$$\sum_{z \in \Z{n}}\Lambda_{a,b,z}^{i}\Lambda_{a,z,c}^{i}\Lambda_{z,b,c}^{i}=\Lambda_{a,b,c}^{i+1}>0$$
 \end{proof}
Hence, by Corollary $\ref{hogy2}$, the main result will be proven if we can show the following:
\begin{theorem}
\label{main2}
Let $A$ be a random subset of $\Z{N}$ where each $x \in \Z{n}$ is chosen independently with probability $p=N^{-\frac{1}{2} + \epsilon}$ for any $\epsilon>0$. Let $\Lambda_{a,b,c}^{i}$ with $(a,b,c)$ distinct be the family of polynomials defined above. Then there exists some $i=i(\epsilon)$ such that:
\begin{align*}
\prob{\Lambda_{a,b,c}^{i}>0\; \text{for all}\; (a,b,c)}= 1-o_{N}(1) 
\end{align*}
\end{theorem}
The following sections aim to prove Theorem $\ref{main2}$. Before we start tackling the proof, however, we need to introduce a powerful tool developed by V. Vu concerning the concentration of multivariate Boolean polynomials.

\section{Boolean polynomials}

Boolean polynomials are objects that arise very naturally in probabilistic combinatorics as a method of counting `small' structures. They often turn out to be highly concentrated around their means.
 
The classical setting is that of random graphs $G(n,p)$ on  the vertex $[n]$, where each edge $ij$ 
in the graph is chosen independently with probability $p$. In this instance our boolean variables are given by $t_{ij}$ for $i<j$ where $t_{ij}=1$ if the edge $ij$ is in the graph and $0$ otherwise. If we are interested in counting the number of copies of a given small graph $K$ we can look at its corresponding Boolean polynomial; for example, the number edges in $G$ (which we may think of as the number of copies of the graph with 2 vertices and one edge) is given by $Y= \sum_{i<j}t_{ij}$. In this very particular case we may use a well known result of Chernoff which states:

\begin{theorem}[Chernoff]
Let $Y=\sum_{1}^{N} t_{i}$ where $t_{i}$ are independent Bernoulli($p$) random variables. Then for any $\lambda>0$
$$\prob{|Y-\exn{}{Y}|>\sqrt{\lambda N}}\leq 2e^{-\frac{\lambda}{4}}$$
\end{theorem}

Of course, one could readily apply this whenever the polynomial we are considering has the property that all terms in the summand are independent, but this is not usually the case. 
This bound was generalised by Azuma:
\begin{theorem}[Azuma]
Let $\exn{j}{Y}=\exn{}{Y|t_{1},\ldots,t_{j}}$ and set $d_{j}(t_{1},\ldots,t_{j})=\exn{j}{Y}-\exn{j-1}{Y}$.
Then for any $\lambda>0$
$$\prob{|Y-\exn{}{Y}|>\lambda}\leq 2 \exp{\Big (-\frac{\lambda^{2}}{2\sum_{1}^{n}\|d_{i}\|^{2}_{\infty}}\Big )}$$
where $\|d_{i}\|_{\infty}$ is the maximum value of $d_{i}$ over all possible values of $t_{1},\ldots,t_{i}$.
\end{theorem}

In the case where the $\|d_{i}\|_{\infty}$ are very small in comparison to the expectation Azuma's bound is ideal for showing strong concentration. Unfortunately, there are many examples where this fails.

Suppose we are interested in the number of triangles in $G$. That is we are interested in the value of the random variable  $Y=\sum_{i<j<k}t_{ij}t_{jk}t_{ik}$.
The expectation of $Y$ is $\Theta(N^{3}p^{3})$, and each edge lies in at most $n-2$ triangles so in this case $\|d_{i}\|_{\infty}\leq n-2$ for all $i$ but more importantly, no matter which ordering we choose for the variables $t_{ij}$, we have that $\|d_{\textit{last}}\|_{\infty}=(1-p)(n-2)$. Azuma's bound will only yield useful information if $\exn{}{Y} >> n^{3/2}$ so if $p>>n^{-1/2}$. On the other hand it is a very rare event that a particular edge lies in $n-2$ triangles: the expected number of such triangles is about $n^{2}p^{2}$. Furthermore we can use Chernoff to show that this number is strongly concentrated around its mean. The intuition here is that one should not look at the maximum value of $d_{i}$ but  rather at its expectation to obtain concentration for smaller ranges of $p$. This is essentially what Vu's result manages to achieve but before we state the theorem we need to introduce some terminology.

The first thing to point out is that since we are dealing with $\{0,1\}$ valued variables the polynomials $t_{1}^{2}t_{2}t_{3}^{3}$ and $t_{1}t_{2}t_{3}$ take identical values, so without loss of generality we may assume that each variable has degree at most $1$.  To be more precise, every Boolean polynomial $P$ has a unique reduced
form as 
 $$P[t_1,\ldots,t_N]=\sum_{B \in \mathcal{A} }w(B)t_B$$
where $\mathcal{A}$ is a family of subsets of $[N]$, $t_B= \prod_{i\in B}t_i$ and $w: \mathcal{A}\to \mathbb{R}$ is some weight function. We say that $P$ is positive if $w\geq 0$.

\begin{definition}
Given $C \subset[N]$ and a Boolean polynomial $P=\sum_{B \in \mathcal{A} }t_B$ let 
$$m(C,l;P)=\sum_{\substack{B \in \mathcal{A}\\ C\subseteq B,\; |B|=l}} w(B) $$ 
in other words $m(C,l;P)$ counts the (weighted) number of terms in $P$ of length $l$ containing the term $t_C$. Furthermore let 
$$m(C;P)=\sum_{l}m(C,l;P)$$
and note that $m(\emptyset;P)$ is by definition the total number of terms in $P$.\\
\end{definition}

We would like to replace the  $\|d_{i}\|_{\infty}$ in Azuma's bound by something weaker such as an average. Now we define the quantities that will play this important role.

\begin{definition}
Given any $C\subset [N]$ and $P=\sum_{B \in \mathcal{A} }w(B)t_B$ the partial derivative of $P$ with respect to $C$ is given by the polynomial
 $$\partial_{C}P=\prod_{i\in C} \frac{\partial}{\partial t_{i}}\phantom y P$$
and set

\begin{equation*}
\exn{j}{P}= \max \{\exn{}{\partial_{C}P}: C \subset [N],|C|\geq j \}
\end{equation*}
\end{definition}

It is clear that if one knows the values of $m(C,l;P)$ for all $C, l$ then one may easily compute $\exn{j}{P}$ for all $j$ as 
$$\exn{}{\partial_C P} = \sum_{B \in \mathcal{A}, C\subseteq B}w(B)\exn{}{t_{B\setminus C}}=\sum_l m(C,l;P)p^{l-|C|}$$

In order to illustrate these definitions, let us go back for a moment to the problem of counting triangles in the random graph $G(n,p)$. Recall that we are looking at the polynomial $Y=\sum_{i<j<k}t_{ij}t_{jk}t_{ik}$ that counts the number of triangles.\\

\noindent Set $C=\{12\}$, that is the single edge $12$. Then
$$\partial_{C}Y=\sum_{k>2}t_{1k}t_{2k}$$
\noindent so $m(C;Y)=n-2$ and $\exn{}{\partial C}=p^{2}(n-2)$. The former quantity is the maximal number of triangles containing the edge $12$ and the latter the expected number. More generally,  
if we fix a single edge there are exactly $n-2$ triangles containing it and if we fix 2 edges there is (at most) 1 triangle containing both edges, therefore
\begin{align*}
\exn{}{\partial_{C}Y}&=(n-2)p^{2}&\text{for any $C$ with $|C|=1$}\\
\exn{}{\partial_{C}Y}&\leq p&\text{for any $C$ with $|C|=2$}\\
\exn{}{\partial_{C}Y}&\leq 1&\text{for any $C$ with $|C|\geq 3$}\\
\end{align*}
Hence, provided that $p>>n^{-1}$ 
\begin{eqnarray}
\label{triangle}
\exn{0}{Y}=\exn{}{Y}= {n \choose 3}p^{3}, &\exn{1}{Y}= \max\{(n-2)p^{2},1\}, & \exn{2}{Y}=1
\end{eqnarray}

We are now ready to state Vu's Theorem on the concentration of multivariate Boolean polynomials:

\begin{theorem}[V. Vu]
\label{Vu}
Let $P$ be a positive reduced Boolean polynomial of degree $k$. Then for any positive numbers  $\mathcal{F}_0>\mathcal{F}_1>\ldots>\mathcal{F}_k$ and $\lambda$ satisfying
\begin{itemize}
\item for all $0\le j \leq k, \; \mathcal{F}_j \geq \exn{j}{P}$ 
\item for all $0\le j \leq k, \; \mathcal{F}_j/\mathcal{F}_{j+1} \geq \lambda + 4j\log{N}$
\end{itemize}
there exist some constants $c_{k}$ and $d_{k}$ depending only on $k$ such that the following holds:
\begin{equation*}
\prob{|P-\exn{}{P}|\geq c_k \sqrt{\lambda\mathcal{F}_0\mathcal{F}_1}}\leq d_ke^{-\frac{\lambda}{4}}
\end{equation*}
\noindent where the $\exn{j}{P}$ are defined as above.
\end{theorem}

We will now apply Vu's Theorem to the triangle counting problem in $G(n,p)$. If we let $p=\Theta(n^{-2/3})$ then, from \eqref{triangle}, it follows $\exn{0}{Y}=\Theta(n)$ and that $\exn{1}{Y}$, $\exn{2}{Y}\leq 1$.\\

Now set $\mathcal{F}_{0}= Cn$, $\mathcal{F}_{1}= \sqrt{Cn}$ and $\mathcal{F}_{2}=1$ where $C$ is some sufficiently large constant. It is easy to check that for $\lambda=a \sqrt{n}$, where $a$ is a constant chosen so that $c_{k}\sqrt{\lambda \mathcal{F}_{0}\mathcal{F}_{1}}\leq \epsilon \exn{}{Y}$, then the $\mathcal{F}_{i}$ meet the requirements of the theorem and therefore:
$$\prob{|Y-\exn{}{Y}|> \epsilon \exn{}{Y}}\leq Ce^{-C^{-1}\sqrt{n}}.$$
which is a considerable improvement of the range of $p$ over the bound obtained using Azuma's inequality.

\newpage
\section{The base case}

Before we tackle the proof of Theorem $\ref{main2}$ we will show first a weaker result which will serve as a nice concrete example as well as paving the way for the more general proof.\\ 
\begin{theorem}
\label{base}
Let $A$ be a random subset of $\Z{N}$ where each $x \in \Z{n}$ is chosen independently with probability $p=N^{-\frac{4}{9} + \epsilon}$ for any $\epsilon>0$. 
Then 
\begin{align*}
\prob{\Lambda_{a,b,c}^{1}>0\; \text{for all}\; (a,b,c)}= 1-o_{N}(1) 
\end{align*}
\end{theorem}
One easy observation is that without loss of generality we may assume that $a,b,c$ are distinct as otherwise if, say $a=b$, we would be ultimately trying to show that any induced function $\phi$ must satisfy that $\phi(0+d)=\phi(0)+\phi(d)$ for $d=c-a$, which is trivially true provided that the function $\phi$ is defined at $d$.\\

We wish to apply Vu's to show that Theorem $\ref{base}$ holds. In order to do so we must estimate accurately the valuee of all the different partial derivatives; in this case the simplest approach is to start by giving an explicit expansion of the polynomial $\Lambda^{1}_{a,b,c}$ :
\begin{align*}
\Lambda^{1}_{a,b,c}&= \sum_{z}\Lambda_{a,b,z}^{0}\Lambda_{a,z,c}^{0}\Lambda_{z,b,c}^{0}\\
&=\sum_{x_{1},x_{2},x_{3},z} (x_{1}+a,x_{1}+b,x_{1}+z,x_{2}+a,x_{2}+z,x_{2}+c,x_{3}+z,x_{3}+b,x_{3}+c)
\end{align*}
where here the $9$-tuple $(v_{1},\ldots,v_{9})$ is used to denote the term $t_{v_{1}}\cdots t_{v_{9}}$. \\

Let $B\subset [N]$; to compute the expected value of  $\partial_{B}\Lambda^{1}_{a,b,c}$ one essentially needs to count the number of occurrences of the set $B$ within each of the terms of  $\Lambda^{1}_{a,b,c}$. \\

\noindent To this end we lay out the following set up:
\begin{enumerate}
\item We begin by identifying the monomials in $\Lambda^{1}$ as points in $\Z{N}^{9}$, parametrized by the quadruple $(x_{1},x_{2},x_{3},z)$, via the linear map $\psi : \Z{N}^{4} \to \Z{N}^{9}$ : 
\begin{equation}
\notag
(x_{1},x_{2},x_{3},z) \mapsto (x_{1}+a,x_{1}+b,x_{1}+z,x_{2}+a,x_{2}+z,x_{2}+c,x_{3}+z,x_{3}+b,x_{3}+c)
\end{equation}
\item Pick  a point $(x_{1},x_{2},x_{3},z) \in \Z{N}^{4}$ uniformly at random.\\
\item Ask what the probability is that $B \subset \psi(x_{1},x_{2},x_{3},z)=(v_{1},\ldots,v_{9})$ i.e. $B$ appears as a (not necessarily ordered) subsequence of $(v_{1},\ldots,v_{9})$. 

\end{enumerate}

\noindent If we were in an ideal world, to be able to apply Vu's Theorem directly, we would have that:
\begin{itemize}
\item $\exn{0}{\Lambda^{1}} = \Theta(N^{4}p^{9})$
\item $\exn{1}{\Lambda^{1}}= o(N^{4}p^{9})$
\end{itemize}

\noindent Unfortunately a quick inspection shows that this is not quite the case: suppose we  choose  to take $x_{1}=x_{2}=x_{3}$ then the $9$-tuple above collapses into  $$P= \sum_{x,z}(x+a,x+b,x+c, x+z)$$

It is easy to see that $\exn{}{P} = \Omega(N^{2}p^{4})$  which is bad news as this is of a greater order of magnitude than $N^{4}p^{9}$  as long as $p< N^{-2/5}$. In particular for $p=N^{-\frac{4}{9} + \epsilon}$, provided $\epsilon$ is sufficiently small, $\exn{}{P} \gg \exn{}{\Lambda^{1}}$. \\

\noindent Furthermore if we take $B=\{a,b,c\}$ then 
\begin{equation*}
\exn{}{\partial_{B}\Lambda^{1}}\geq\exn{}{\partial_{B}P}\geq Np
\end{equation*}
which again for $p=N^{-\frac{4}{9} + \epsilon}$ is of a greater order of magnitude than 
$N^{4}p^{9}$, or $N^{2}p^{4}$ for that matter. \\

At this point it may seem as though there is little hope of being able to deduce Theorem $\ref{base}$ by applying Vu's Theorem.
However if we take a closer look into $P$ we note that $\Lambda^{0}_{a,b,c} \geq P$. In particular  $P>0$  implies $\Lambda^{0}_{a,b,c}>0$ and so  
the triangle $[a,b,c]$  was already in fact a simplex in the surface.\\

\begin{remark}
Informally speaking what is happening here is that there are some inherent degenerate terms in the definition
of $\Lambda^{1}$, coming from some specific quadruples $(x_{1},x_{2},x_{3},z)$, which for a certain range of $p$ become bigger than the main term.
At first glance this might not seem problematic, since ultimately we wish to show that $\Lambda^{1}_{a,b,c} >0$ with very high probability -- i.e. decaying exponentially in $N$ so a priori having big, but nonetheless positive, degenerate terms in $P$ should make  the task easier. The  subtlety is that despite having a big expectation $P$ will not be strongly concentrated around its mean and in fact $\prob{P>0} \ll \prob{\Lambda^{1}>0}$. For instance, in the above example if we choose $p=N^{-9/19}$ then  $\exn{}{P}\geq N^{2/19}\gg N^{4}p^{9}$ but $\prob{P>0}\leq \prob{\Lambda^{0}>0}\leq \exn{}{\Lambda^{0}}= N^{-8/19}$.    
\end{remark}

The way to get around this issue is by restricting our attention to quadruples $(x_{1},x_{2},x_{3},z)$ such that the all 9 linear forms take distinct values.\\

\begin{definition}
We will say that a quadruple $(x_{1},x_{2},x_{3},z)$  is degenerate if  it satisfies any non trivial relation $r(x_{1},x_{2},x_{3},z,a,b,c)$  of length at most 4. That is to say, we may find $y_{1}, y_{2}, y_{3}, y_{4}\in\{x_{1},x_{2},x_{3},z,a,b,c\}$ and $\epsilon_{i}\in \{-1,0,1\}$ such that
$$\epsilon_{1}y_{1}+\epsilon_{2}y_{2}+\epsilon_{3}y_{3}+\epsilon_{4}y_{4}=0$$	 
\end{definition}
\noindent Let $\mathcal{H}$ to be the set of all non-degenerate quadruples and set 
\begin{align*}
\tilde{\Lambda}^{1}:= \sum_{(x_{1},x_{2},x_{3},z)\in \mathcal{H}}\psi(x_{1},x_{2},x_{3},z)
\end{align*}

Note that, for any non trivial relation $r$, the number of quadruples that satisfy it is at most $N^{3}$, furthermore the number of such relations is bounded above by an absolute constant $C$ (here setting $C=2^{3}\, 7^{4}$ will do). Therefore the number of degenerate quadruples is at most  $CN^{3}$ and the number of monomials in $\tilde{\Lambda}^{1}$ is still of order $N^{4}$, all of length 9  by definition, as otherwise we would obtain some relation     $r(x_{1},x_{2},x_{3},z,a,b,c)$  of length 4. \\

\noindent Now we can easily compute the expectation of $\tilde{\Lambda}^{1}$:
\begin{equation*}
\exn{}{\tilde{\Lambda}^{1}}= \sum_{\bold v \in \mathcal{H}}\exn{}{t_{\psi(\bold v)}}=\Theta(N^{4}p^{9})
\end{equation*}

\noindent Before we dive any further into computations we introduce a bit of notation: for $B \subset [N]$ and $Q$ any Boolean polynomial let
$$P(B;Q) = \frac{1}{m(\emptyset;Q)}m(B;Q)$$
that is the proportion of monomials in $Q$ containing the term $t_{B}$. \\

In particular, $P(B; \tilde{\Lambda}^{1})$ is precisely the probability that $B$ is contained in a monomial of $\tilde{\Lambda}^{1}$ when we pick a non-degenerate quadruple $(x_{1},x_{2},x_{3},z)$ uniformly at random. \\

\begin{proposition}
Given any $B \subset [N]$ with we have that
\label{dist4}
\begin{equation}
\begin{split}
P(B; \tilde{\Lambda}^{1})&\leq CN^{-\lceil\frac{|B|}{2}\rceil} \quad  \textrm{if} \quad |B|<9 \\
         &\leq CN^{-4} \qquad \; \;\textrm{if}  \quad |B|=9
\end{split}
\end{equation}
where $C$ is some sufficiently large positive absolute constant.
\end{proposition}

\begin{proof}
Recall that 
\begin{equation}
\tilde{\Lambda}_{a,b,c}^{1}\subset \Lambda^{1}_{a,b,c}= \sum_{z}\Lambda_{a,b,z}^{0}\Lambda_{a,z,c}^{0}\Lambda_{z,b,c}^{0}
\end{equation}
where we write $P\subset Q$ to mean that every monomial in $P$ is also a monomial in $Q$.\\

\noindent Hence:

\begin{equation}
m(B;\tilde{\Lambda}^{1})\leq \sum_{z}\sum_{S_{1}\sqcup S_{2}\sqcup S_{3}= B}m(S_{1};\Lambda_{a,b,z}^{0})m(S_{2};\Lambda_{a,z,c}^{0})m(S_{3};\Lambda_{z,b,c}^{0})
\end{equation}
where we write $S_{1}\sqcup S_{2}\sqcup S_{3}= B$ to mean that $S_{1},S_{2},S_{3}$ partition $B$. Furthermore, since the number of terms in $\tilde{\Lambda}^{1}$  is at least  $\frac{1}{2}N^{4}$ for sufficiently large $N$  we certainly have:

\begin{equation}
\label{eq2}
P(B;\tilde{\Lambda}^{1})\leq \frac{2}{N}\sum_{z}\sum_{S_{1}\sqcup S_{2}\sqcup S_{3}= B}P(S_{1};\Lambda_{a,b,z}^{0})P(S_{2};\Lambda_{a,z,c}^{0})P(S_{3};\Lambda_{z,b,c}^{0})
\end{equation}

\noindent Fix a partition $S_{1}\sqcup S_{2}\sqcup S_{3}= B$ and note that the number of all such partitions is bounded by a constant $C$ independent of $N$; $C=|B|^{6}\leq 9^{6}$ will do. \\

\noindent We look at $P(S_{1};\Lambda_{a,b,z}^{0})$ for a fixed $z\in \Z{N}$. Each monomial in $\Lambda_{a,b,z}^{0}$ is given by the triple $(x+a,x+b,x+z)$ where $x$ ranges over all values in $\Z{N}$. \\

\noindent It follows that
\begin{equation}
\begin{split}
\label{eq1}
P(S_{1};\Lambda_{a,b,z}^{0}) &\leq C'N^{-\lceil\frac{|S_{1}|}{2}\rceil} \quad  \textrm{if} \quad |S_{1}|<3 \\
         &\leq C'N^{-1} \qquad \; \; \;\textrm{if}  \quad |S_{1}|=3
\end{split}
\end{equation}
and similarly for  $P(S_{2};\Lambda_{a,z,c}^{0})$ and $P(S_{3};\Lambda_{z,b,c}^{0})$.\\

\begin{remark}
 The alert reader would have noticed that the above is a rather convoluted way of saying  that $P(S_{1};\Lambda_{a,b,z}^{0})\leq CN^{-1}$ provided $S_{1}\neq \emptyset$. The reason for presenting it in this way is to make the arguments analogous to those in the general case.\\
\end{remark}
\noindent We split the sum in $\ref{eq2}$ into two parts and consider each case individually:\\

\noindent \case{1} $|S_{j}|<3$ for all $j=1,2,3$ : \\

\noindent Using the inequality in $\ref{eq1}$ on the expression $\ref{eq2}$:
\begin{equation*}
\begin{split}
P(B;\tilde{\Lambda}^{1})&\leq \frac{1}{N}\sum_{z}\sum_{S_{1}\sqcup S_{2}\sqcup S_{3}= B}C'N^{-\lceil\frac{|S_{1}|}{2}\rceil} C'N^{-\lceil\frac{|S_{2}|}{2}\rceil} C'N^{-\lceil\frac{|S_{3}|}{2}\rceil}\\
&\leq \sum_{S_{1}\sqcup S_{2}\sqcup S_{3}= B}(C')^{3}N^{-\lceil\frac{|S_{1}|+|S_{2}|+|S_{3}|}{2}\rceil}\\
&\leq C(C')^{3}N^{-\lceil\frac{|B|}{2}\rceil}
\end{split}
\end{equation*}
where the second inequality comes from the fact that there are at most $C$ partitions of $B$ as $S_{1}\sqcup S_{2}\sqcup S_{3}$. \\

\noindent \case{2} $|S_{j}|=3$ for some $j=1,2,3$:\\

\noindent Without loss of generality assume that $|S_{1}|=3$.\\

\begin{claim} $P(S_{1};\Lambda_{a,b,z}^{0})= 0$
for all but at most $6$ values of $z \in \Z{n}$.\\

\noindent Let $S_{1}=\{y_{1},y_{2},y_{3}\}$. Write $\Lambda_{a,b,z}^{0}=\sum_{x}(x+a,x+b,x+z)$ and suppose we decide to put the assignment
\begin{flalign*}
x+a &= y_{1}\\
x+b &= y_{2}\\
x+z &= y_{3}
\end{flalign*}
Then $x=y_{1}-a$ which implies $z=y_{3}+a-y_{1}$. Thus for any given assignment $z$ is fully determined and there  is a total of $3!=6$ possible assignments. The claim follows.\\
\end{claim}

\noindent Hence from $\ref{eq2}$ again:
\begin{equation*}
P(B;\tilde{\Lambda}^{1})\leq \frac{6}{N}\sum_{S_{1}\sqcup S_{2}\sqcup S_{3}= B}(C'N^{-1})P(S_{2};\Lambda_{a,z,c}^{0})P(S_{3};\Lambda_{z,b,c}^{0})
\end{equation*}
Now  if $|S_{2}|=|S_{3}|=3$ then 
\begin{equation*}
P(B;\tilde{\Lambda}^{1})\leq \frac{6}{N}\sum_{S_{1}\sqcup S_{2}\sqcup S_{3}= B}(C'N^{-1})^{3}\leq 6C(C')^{3}N^{-4}
\end{equation*}
which covers the case $|B|=9$.\\ 

\noindent Otherwise without loss of generality we may assume that $|S_{3}|<3$ which implies $P(S_{2};\Lambda_{a,z,c}^{0})P(S_{3};\Lambda_{z,b,c}^{0})\leq (C')^{2}N^{-\lceil\frac{|S_{2}|}{2}\rceil-\lceil\frac{|S_{3}|}{2}\rceil+1}$ . Thus

\begin{equation*}
\begin{split}
P(B;\tilde{\Lambda}^{1})&\leq\sum_{S_{1}\sqcup S_{2}\sqcup S_{3}= B} 6(C')^{3}N^{-\lceil\frac{|S_{2}|}{2}\rceil-\lceil\frac{|S_{3}|}{2}\rceil-1}\\
&\leq 6C(C')^{3}N^{-\lceil\frac{|S_{1}|+|S_{2}|+|S_{3}|}{2}\rceil}
\end{split}
\end{equation*}
as required.\\
\end{proof}
\begin{proof}[Proof of Theorem $\ref{base}$]
\noindent  Once we have proven Proposition $\ref{dist4}$ it is an easy task to compute $\exn{j}{\tilde{\Lambda}^{1}}$:
\begin{equation*}
\exn{}{\partial_{B}\tilde{\Lambda}^{1}}\leq CN^{4}P(B;\tilde{\Lambda}^{1})p^{9-|B|}\leq CN^{4-\lceil\frac{B}{2}\rceil}p^{9-|B|}
\end{equation*}
and therefore, since $Np^{2}\gg 1$:

\begin{align*}
\exn{0}{\tilde{\Lambda}^{1}}&\leq CN^{4}p^{9}\\
&\cdots\\
\exn{j}{\tilde{\Lambda}^{1}}&\leq CN^{4}p^{9}(Np^{2})^{-\lceil\frac{j}{2}\rceil}\\
&\cdots\\
\exn{9}{\tilde{\Lambda}^{1}}&\leq C\\
\end{align*}

\noindent We are now well placed to apply Vu's Theorem. Setting $\mathcal{F}_{0}=CN^{4}p^{9}$,\; $\mathcal{F}_{j+1}=N^{-1/18}\mathcal{F}_{j}$
and $\lambda= cN^{1/18}$ where the constants are $c$ and $C$  are appropriately chosen the conditions of Vu's theorem are satisfied and
$$\prob{|\tilde{\Lambda}_{a,b,c}^{1}-\exn{}{\tilde{\Lambda}_{a,b,c}^{1}}|> \epsilon \exn{}{\tilde{\Lambda}^{1}_{a,b,c}}}\leq Ce^{-C^{-1}N^{1/18}}$$

\noindent Moreover
$$\prob{\Lambda_{a,b,c}^{1}>0\; \text{for all}\; (a,b,c)}\leq 1-N^{3}Ce^{-C^{-1}N^{1/18}}=1-o_{N}(1)$$ 
\end{proof}
\section{The general case}
Most of the arguments above can be extended to the proof of Theorem $\ref{main2}$, the main obstacle being  that in the general case we cannot simply expand out $\Lambda_{a,b,c}^{i}$ to see what it looks like and extract the properties we are interested in. However we can go around this difficulty by exploiting the recursive nature of this polynomial which is very amenable to inductive arguments. \\

\begin{proposition}  
\label{map}
For each $i\geq 0$ there exists a linear map $\psi^{i}: \Z{N}^{d+3} \to \Z{N}^{2d+1}$ such that 
$$\Lambda^{i}_{a,b,c}= \sum_{{\bold v}\in \Z{N}^{d}}t_{\psi^{i}((a,b,c)\oplus\bold v)}$$
where  $2d+1=deg(\Lambda_{a,b,c}^{i})= 3^{i+1}$\\
\noindent Furthermore we may choose $\psi^{i+1}$ in such way that for $\bold v_{1},\bold v_{2},\bold v_{3} \in \Z{N}^{d}$ and $z \in \Z{N}$
\begin{equation}
\label{eq1}
(a,b,c)\oplus\bold v_{1}\oplus \bold v_{2} \oplus \bold v_{3} \oplus z \longmapsto \psi_{a,b,z}^{i}(\bold v_{1})\oplus \psi_{a,z,c}^{i}(\bold v_{2})\oplus\psi_{z,b,c}^{i}(\bold v_{3})
\end{equation}
\end{proposition}
\noindent For the sake of convenience we shall use $\psi^{i}_{a,b,c}(\bold v)$ to denote $\psi^{i}((a,b,c)\oplus\bold v)$ \\

\begin{proof}
We argue by induction on $i$.  We have already seen in the previous section that this holds in the case $i=1$ so assume that the result holds for $i$. 

\noindent By the induction hypothesis:
\begin{equation*}
\begin{split}
\Lambda^{i+1}_{a,b,c}&=\sum_{z}\Lambda_{a,b,z}^{i}\Lambda_{a,z,c}^{i}\Lambda_{z,b,c}^{i}\\
&=\sum_{z}\Big (\sum_{{\bold v_{1}}\in \Z{N}^{d}}\psi^{i}_{a,b,z}(\bold v_{1}) \Big )\Big (\sum_{{\bold v_{2}}\in \Z{N}^{d}}\psi^{i}_{a,z,c}(\bold v_{2}) \Big )\Big (\sum_{{\bold v_{3}}\in \Z{N}^{d}}\psi^{i}_{z,b,c}(\bold v_{3}) \Big )\\
&=\sum_{\substack{\bold v_{1},\bold v_{2},\bold v_{3} 
\in \Z{N}^{d}\\ z\in \Z{N}}}\psi^{i}_{a,b,z}(\bold v_{1})\psi^{i}_{a,z,c}(\bold v_{2})\psi^{i}_{z,b,c}(\bold v_{3}) 
\end{split}
\end{equation*}
Thus taking $\psi^{i+1}_{a,b,c}$ to be as in equation \eqref{eq1} will do. All that remains to show that it is linear:
\begin{multline*}
\psi^{i+1}_{a+a',b+b',c+c'}(\bold v_{1} +\bold v'_{1}\oplus \bold v_{2}+\bold v'_{2} \oplus \bold v_{3}+\bold v'_{3} \oplus z + z' )=\\
\shoveleft  \psi_{a+a',b+b',z+z'}^{i}(\bold v_{1}+ \bold v'_{1})\oplus \psi_{a+a',z+z',cc+c'}^{i}(\bold v_{2}+\bold v'_{2})\oplus\psi_{z+z',b+b',c+c'}^{i}(\bold v_{3}+\bold v_{3})
 \end{multline*}
Now looking at each component individually;
\begin{equation*}
\psi_{a+a',b+b',z+z'}^{i}(\bold v_{1}+ \bold v'_{1})= \psi^{i}_{a,b,c}(\bold v_{1})+\psi^{i}_{a',b',c'}(\bold v'_{1})
\end{equation*}
by the induction hypothesis and similarly for $\bold v_{2}, \bold v_{3}$. Hence the above is the same as
\begin{equation*}
\psi_{a,b,z}^{i}(\bold v_{1})\oplus \psi_{a,z,c}^{i}(\bold v_{2})\oplus\psi_{z,b,c}^{i}(\bold v_{3})+\psi_{a',b',z'}^{i}(\bold v'_{1})\oplus \psi_{a',z',c'}^{i}(\bold v'_{2})\oplus\psi_{z',b',c'}^{i}(\bold v'_{3})
\end{equation*}
and hence $\psi^{i+1}$ is linear.


\end{proof}

As in the previous section, we cannot  hope to show that the polynomial $\Lambda^{i}_{a,b,c}$ is strongly concentrated around its expectation because of the presence of `noise terms'. We need introduce an analogous notion of degeneracy that gets rid of these terms whilst keeping the expectation of the right order. Before we do so however we need to prove one more property of $\psi^{i}_{a,b,c}$.
\begin{proposition}
\label{form}
Let us write $\psi^{i}:\Z{N}^{d+3}\to \Z{N}^{2d+1}$ as $(L_{1},L_{2},\ldots,L_{2d+1})$ where $L_{j}: \Z{N}^{d+3}\to \Z{N}$ are linear forms. Then for $\bold v = (a,b,c)\oplus (v_{1},v_{2},\ldots, v_{d})$ 
$$L_{j}(\bold v)= x_{j}+y_{j}$$ 
for $x_{j} \in \{a,b,c,v_{1}\ldots, v_{d}\}$ and $y_{j}\in \{v_{1},\ldots,v_{d}\}$. Furthermore $x_{j}\neq y_{j}$ and $L_{j}\neq L_{j'}$ for $j\neq j'$.
\end{proposition}
\begin{proof}
Unsurprisingly we argue by induction on $i$. Set
\begin{flalign*}
(L_{1},L_{2},\ldots,L_{2d+1})&=\psi^{i}((a,b,z)\oplus \bold v_{1})\\
(L'_{1},L'_{2},\ldots,L'_{2d+1})&=\psi^{i}((a,z,c)\oplus \bold v_{2})\\
(L''_{1},L''_{2},\ldots,L''_{2d+1})&=\psi^{i}((z,b,c)\oplus \bold v_{1})
\end{flalign*}
Then
\begin{equation*}
\psi_{a,b,c}^{i+1}(\bold v_{1}\oplus\bold v_{2}\oplus\bold v_{3}\oplus z)=(L_{1},\ldots,L_{2d+1})\oplus(L'_{1},\ldots,L'_{2d+1})\oplus(L''_{1},\ldots,L''_{2d+1})
\end{equation*}
thus, since each of the linear maps in each component has the desired form by induction, we are done.

Also by induction we cannot have that $\{L_{j}: 1\leq j\leq 2d+1 \}$ are distinct and similarly for the $L'_{j}$ and $L''_{j}$. The only possible equalities are therefore of the form $L_{j}=L'_{j'}$ for instance but this can only happen if $L_{j}$ is one of $a+b$, $b+c$ or $a+c$ and this is not possible by the induction hypothesis.\\
\end{proof}
\noindent We can now define our notion of degeneracy:\\

\begin{definition}
\emph{A $d$-tuple $\bold v = (v_{1},v_{2},\ldots, v_{d})$  is degenerate if  it satisfies any non trivial relation $r(v_{1},\ldots,v_{2},\ldots,v_{d},a,b,c)$  of length at most 4. }\\
\end{definition}

\noindent Let $\mathcal{H}^{i}_{a,b,c}\subset \Z{N}^{d}$ be the set of all non-degenerate $d$-tuples and set 

\begin{align*}
\tilde{\Lambda}^{i}_{a,b,c}:= \sum_{{\bold v}\in \mathcal{H}^{i}_{a,b,c}}\psi^{i}_{a,b,c}(\bold v)
\end{align*}

The number of such relations is bounded by an absolute constant $C< 8(d+3)^{4}$ independent of $N$ and the number of $d$-tuples satisfying  a fixed relation $r$ is $O(N^{3d})$ . Therefore $|\mathcal{H}_{i}|=\Omega(N^{d})$. Furthermore it follows from Proposition $\ref{form}$ that for all $\bold v \in \mathcal{H}_{i}$, all $2d+1$ coordinates of $\psi^{i}_{a,b,c}(\bold v)$ are distinct. Hence
\begin{equation*}
\exn{}{\tilde{\Lambda}^{i}_{a,b,c}}\geq CN^{d}p^{2d+1}
 \end{equation*}

\begin{proposition}
\label{dist3}
Given any $B \subset [N]$ with $|B|\leq 2d+1$ we have that

\begin{equation}
\begin{split}
P(B;\tilde{\Lambda}^{i}_{a,b,c}) &\leq CN^{-\lceil\frac{|B|}{2}\rceil} \quad \quad \textrm{if} \quad |B|<2d+1 \\
         &\leq CN^{-d} \qquad \quad \; \; \textrm{if}  \quad |B|=2d+1
\end{split}
\end{equation}
where we recall the probability measure $P$ is defined as $$P(B;\tilde{\Lambda}^{i}_{a,b,c}):=\frac{1}{m(\emptyset;\tilde{\Lambda}^{i}_{a,b,c})}m(B;\tilde{\Lambda}^{i}_{a,b,c})$$
\end{proposition}

\begin{proof}
We will argue by induction on $i$. The case $i=1$ was handled in the previous section so assume that the proposition holds for $i$. 
Note that if $\bold v = \bold v_{1}\oplus \bold v_{2} \oplus \bold v_{3}\oplus z$ is in $\mathcal{H}^{i+1}_{a,b,c}$ then we must have that $\bold v_{1} \in \mathcal{H}^{i}_{a,b,z}$, $\bold v_{2}\in\mathcal{H}^{i}_{a,z,c}$ and $\bold v_{3}\in\mathcal{H}^{i}_{z,b,c}$, i.e
\begin{equation*}
\mathcal{H}^{i+1}_{a,b,c}\subseteq \mathcal{H}^{i}_{a,b,z}\oplus  \mathcal{H}^{i}_{a,z,c}\oplus \mathcal{H}^{i}_{z,b,c}
\end{equation*}
Therefore
\begin{equation}
\tilde{\Lambda}^{i}_{a,b,c} \subseteq \sum_{z}\tilde{\Lambda}^{i-1}_{a,b,z}\cdot \tilde{\Lambda}^{i-1}_{a,z,c}\cdot \tilde{\Lambda}^{i-1}_{z,b,c}
\end{equation}\\
\noindent Now fix a partition $S_{1}\sqcup S_{2}\sqcup S_{3}= B$. Again the number of such partitions is also bounded by a constant $C$ independent of $N$, a conservative bound would be $C< |B|^{6}\leq (2d+1)^{6}$.
The clearly we have that 
\begin{equation*}
m(B,\tilde{\Lambda}^{i+1}_{a,b,c})\leq C\sum_{z}m(S_{1},\tilde{\Lambda}^{i}_{a,b,z})m(S_{2},\tilde{\Lambda}^{i}_{a,z,c})m(S_{3},\tilde{\Lambda}^{i}_{z,b,c})
\end{equation*}
and thus
\begin{equation}
\label{eq3}
P(B;\tilde{\Lambda}^{i+1}_{a,b,c}) \leq C\frac{1}{N}\sum_{z}P(S_{1};\tilde{\Lambda}^{i}_{a,b,z})P(S_{2};\tilde{\Lambda}^{i}_{a,z,c})P(S_{3};\tilde{\Lambda}^{i}_{z,b,c})
\end{equation}
We need to consider three cases:\\

\noindent \case{1} $|S_{j}|<k=2d+1$ for all $j=1,2,3$.

\noindent By the induction hypothesis $P(S_{j};\tilde{\Lambda}^{i-1})\leq  CN^{-\left \lceil\frac{|S_{j}|}{2}\right \rceil}$ and hence by \eqref{eq3}
\begin{equation}
\begin{split}
P(B;\tilde{\Lambda}^{i}_{a,b,c})&\leq CN^{-\left \lceil\frac{|S_{1}|}{2}\right \rceil}N^{-\left \lceil\frac{|S_{2}|}{2}\right \rceil}N^{-\left \lceil\frac{|S_{3}|}{2}\right \rceil}\\
&\leq CN^{-\left \lceil\frac{|B|}{2}\right \rceil}
\end{split}
\end{equation}
as desired.\\

\noindent \case{2} $|B|<2d+1$ and $|S_{j}|=k$ for some $j=1,2,3$\\

\noindent Without loss of generality we may assume that  $|S_{1}|= k$. \\

\begin{claim}
$P(S_{1};\tilde{\Lambda}^{i}_{a,b,z})=0$ for all but at most $k!$ values of $z$.
\end{claim}

\begin{proof}
It can be shown by induction that the map $\bold v\oplus z\mapsto \psi^{i}	_{a,b,z}(\bold v)$ is injective. Let $S=\{s_{1},s_{2},\ldots,s_{k}\}$, then for each $\pi: [k] \to [k]$ a permutation, the equation 
$$\psi_{a,b,z}(\bold v)= (s_{\pi(1)},s_{\pi(2)},\ldots,s_{\pi(k)})$$
has at most $1$ solution, in particular at most one possible value of $z$. There is  a total $k!$ permutations and hence the claim follows. 
\end{proof}
\noindent By the induction hypothesis $P(S_{1};\tilde{\Lambda}^{i}_{a,b,z})\leq CN^{\frac{k-1}{2}}$an so looking back at equation \eqref{eq3} we have:
$$P(B;\tilde{\Lambda}^{i+1}_{a,b,c}) \leq CN^{\frac{-k+1}{2}-1}P(S_{2};\tilde{\Lambda}^{i}_{a,z,c})P(S_{3};\tilde{\Lambda}^{i}_{z,b,c})$$
On the other hand, since one of $|S_{2}|$,$|S_{3}|$ is smaller than $k$ we certainly have the inequality $P(S_{2};\tilde{\Lambda}^{i}_{a,z,c})P(S_{3};\tilde{\Lambda}^{i}_{z,b,c})\leq CN^{-\left \lceil\frac{|S_{2}|}{2}\right \rceil-\left \lceil\frac{|S_{3}|}{2}\right \rceil+1}$
Substituting back in the above:
\begin{equation*}
\begin{split}
P(B;\tilde{\Lambda}^{i+1}_{a,b,c})&\leq CN^{\frac{-k+1}{2}}N^{-\left \lceil\frac{|S_{2}|+|S_{3}|}{2}\right \rceil}\\
&\leq CN^{-\left \lceil\frac{k+|S_{2}|+|S_{3}|}{2}\right \rceil}=CN^{-\left \lceil\frac{|B|}{2}\right \rceil}
\end{split}
\end{equation*}\\

\noindent \case{3} $|S_{1}|=|S_{2}|=|S_{3}|=k$. Then $|B|=3k=6d+3$ and the above yields
\begin{equation*}
P(B;\tilde{\Lambda}^{i+1}_{a,b,c})\leq CN^{-1}N^{-d}N^{-d}N^{-d}=CN^{-3d-1}
\end{equation*}
which completes the proof.\\
\end{proof}

\noindent We can now easily estimate $\exn{j}{\tilde{\Lambda}^{i}_{a,b,c}}$ :\\

\begin{proposition} Suppose that $Np^{2}\gg 1$, then for $j=1,2,\ldots, 2d$ there exists an absolute constant $C$ depending only on $d$ such that
\label{dist2} 
\begin{equation}
\exn{j}{\tilde{\Lambda}^{i}_{a,b,c}}\leq C(Np^{2})^{-\left \lceil\frac{j}{2}\right \rceil}N^{d}p^{2d +1}
\end{equation}
and $\exn{2d+1}{\tilde{\Lambda}^{i}_{a,b,c}}\leq C$.
\end{proposition}

\begin{proof}
Let $B\subset [N]$ such that $|B|=j< 2d+1$. By Proposition $\ref{dist3}$,
\begin{equation*}
\begin{split}
\exn{}{\partial_{B}\tilde{\Lambda}^{i}_{a,b,c}}&\leq CP(B;\tilde{\Lambda}^{i}_{a,b,c})N^{d}p^{2d+1-|B|}\\
&\leq CN^{-\lceil\frac{|B|}{2}\rceil}N^{d}p^{2d+1-|B|}\\
&\leq C(Np^{2})^{-\left \lceil\frac{|B|}{2}\right \rceil}N^{d}p^{2d +1}\\
&=C(Np^{2})^{-\left \lceil\frac{j}{2}\right \rceil}N^{d}p^{2d +1}\; .
\end{split}
\end{equation*}
If $|B|= 2d+1$ then $\exn{}{\partial_{B}\tilde{\Lambda}^{i}_{a,b,c}}\leq CP(B;\tilde{\Lambda}^{i}_{a,b,c})N^{d}\leq CN^{-d}N^{d}=C$
\end{proof}
\noindent Since the above inequality holds for any $B\subset [N]$ and $Np^{2}\gg 1$ equation \eqref{dist2} follows.\\

\noindent We are now ready to prove our main result:

\begin{proof}[Proof of Theorem $\ref{main2}$]

First we choose $i$ sufficiently large so that $3^{i+1}=2d+1=k\geq \frac{\epsilon +1}{2\epsilon}$. Then $\exn{}{\tilde{\Lambda}^i}= N^{d}p^{2d+1}=N^{\frac{k-1}{2}}p^k = N^{k\epsilon -1/2} \geq N^\epsilon$. \\
Also note that $Np^{2}=N^{\epsilon}$ and thus by Proposition $\ref{dist2}$ it follows that $\exn{j}{\tilde{\Lambda}^{i}}\leq C_k$ for $j\geq1$ for some absolute constant depending only on $k$.\\

\noindent Now we
apply Vu's Theorem $(\ref{Vu})$, setting $\lambda=N^{\frac{\epsilon}{2k}}$ and
\begin{align*}
\mathcal{F}_0=&C_{k}\exn{}{\tilde{\Lambda}^i}=C_kN^{\epsilon} \\
\mathcal{F}_{j+1}=& N^{-\epsilon/k}\mathcal{F}_{j}=C_kN^{\frac{k-j}{k}\epsilon}\;.\\
\end{align*}

Clearly $\mathcal{F}_j \geq \exn{j}{\tilde{\Lambda}^i}$ and, provided $N$ is large enough, $\mathcal{F}_j/\mathcal{F}_{j+1} =N^{\epsilon/k}\geq \lambda + 4j\log{N}$.
Hence by Vu's Theorem,
\begin{equation}
\prob{|\tilde{\Lambda}^{i}_{a,b,c}-\exn{}{\tilde{\Lambda}^{i}_{a,b,c}}|\geq c_k \sqrt{\lambda\mathcal{F}_0\mathcal{F}_1}}\leq d_ke^{-\frac{\lambda}{4}}
\end{equation}
where again $c_{k}$ and $d_{k}$ are constants dependent only on $k$. Now $ c_k \sqrt{\lambda\mathcal{F}_0\mathcal{F}_1}=C_{k}c_{k}N^{\epsilon(1-1/4k)}\leq N^{\epsilon(1-1/8k)}$ and $d_{k}e^{-\frac{\lambda}{4}}\leq e^{-N^{\epsilon/4k}}$ again provided $N$ is sufficiently large. \\

\noindent Thus,
\begin{equation}
\prob{|\tilde{\Lambda}^{i}_{a,b,c}- N^{\epsilon}| \geq N^{\epsilon(1 - 1/8k)} }\leq e^{-N^{\epsilon/4k}}
\end{equation}
provided $N$ is large enough.   
In particular it follows that $\prob{\tilde{\Lambda}^i = 0} \leq e^{-N^{\epsilon/4k}}$ and therefore:
\begin{align*}
\prob{ \tilde{\Lambda}^{i}_{a,b,c} > 0 \; \text{for all}\;(a,b,c)} \geq& 1-N^2e^{-N^{\epsilon/4k}}\\
=& 1-o_N(1)
\end{align*}
\end{proof}

\nocite{*}
\bibliography{Freimanhomprof}
\bibliographystyle{plain}

\end{document}